\documentclass[oneside,english]{amsart}
\usepackage[T1]{fontenc}
\usepackage[latin9]{inputenc}
\usepackage{amsthm}
\usepackage{amstext}
\usepackage{amssymb}
\usepackage{esint}
\usepackage{graphicx}
\usepackage{enumerate}

\usepackage{graphicx}
\graphicspath{{noiseimages/}}

\makeatletter
\newtheorem{theorem}{Theorem}[section]
\newtheorem{lemma}[theorem]{Lemma}
\newtheorem{proposition}[theorem]{Proposition}
\newtheorem{corollary}[theorem]{Corollary}
\numberwithin{figure}{section}
\theoremstyle{definition}

\theoremstyle{remark}
\newtheorem{remark}[theorem]{Remark}

\numberwithin{equation}{section}
\makeatother

\usepackage{babel}

	\DeclareMathOperator{\HI}{HI}
	
	\DeclareMathOperator{\dist}{dist}
	\DeclareMathOperator{\loc}{loc}

	\DeclareMathOperator{\cp}{cap}
	
\begin{document}

\title[Spectral Properties]{Integral estimates of conformal derivatives and spectral properties of the Neumann-Laplacian}

\author{V.~Gol'dshtein, V.~Pchelintsev, A.~Ukhlov}

\begin{abstract}
In this paper we study integral estimates of derivatives of conformal mappings $\varphi:\mathbb D\to\Omega$ of the unit disc $\mathbb D\subset\mathbb C$ onto bounded domains $\Omega$ that satisfy the Ahlfors condition. These integral estimates lead to estimates of constants in Sobolev-Poincar\'e inequalities, and by the Rayleigh quotient we obtain spectral estimates of the Neumann-Laplace operator in non-Lipschitz domains (quasidiscs) in terms of the (quasi)conformal geometry of the domains. Specifically, the lower estimates of the first non-trivial eigenvalues of the Neumann-Laplace  operator in some  fractal type domains (snowflakes) were obtained.
\end{abstract}
\maketitle

\section{Introduction}

\subsection{Estimates of Conformal Derivatives}

In the work \cite{GU16} we obtained lower estimates of the first non-trivial eigenvalues of the Neumann-Laplace operator in the terms of integrals of complex derivatives (i.~e. hyperbolic metrics) of conformal mappings $\varphi:\mathbb D\to\Omega$. 
Let us recall that the classical Koebe distortion theorem \cite{Bieb16} gives the following estimates of the complex derivatives in the case of univalent analytic functions (conformal homeomorphisms): 
$\varphi:\mathbb D\to\Omega$ normalized so that $\varphi(0)=0$ and $\varphi'(0)=1$:
$$
\frac{1-|z|}{(1+|z|)^3}\leq |\varphi'(z)|\leq \frac{1+|z|}{(1-|z|)^3}.
$$

The example of the Koebe function 
$$
\varphi(z)=\frac{z}{(1-z)^2}
$$
which maps the unit disc $\mathbb D$ onto $\Omega=\mathbb C\setminus (-\infty,-1/4]$ shows that these estimates don't give even square integrability of the complex derivatives in arbitrary simply connected planar domains. But if $\Omega\subset\mathbb C$ is a simply connected planar domain of finite measure then by simple calculation
$$
\iint\limits_{\mathbb D} |\varphi'(z)|^2~dxdy= \iint\limits_{\mathbb D} J(z, \varphi)~dxdy=|\Omega|<\infty.
$$
(We identify the complex plane $\mathbb C$ and the real plane $\mathbb R^2$:  $\mathbb C\ni z=x+iy=(x,y)\in \mathbb R^2$.)

Hence in special classes of domains we have better integral estimates of the complex derivatives than by the Koebe distortion theorem. 

In the present work we study integral  estimates of the complex derivatives in domains bounded by Jordan curves that satisfy the Ahlfors three point condition \cite{Ahl}. For this study we introduce a notion of hyperbolic (integral) $\alpha$-dilatation of $\Omega$:
$$
Q(\alpha, \Omega):=\iint\limits_{\mathbb D} |\varphi'(z)|^{\alpha}~dxdy=
\iint\limits_{\Omega}\left|\left( \varphi^{-1}\right)'(w)\right|^{2-\alpha}dudv.
$$

The finiteness of the hyperbolic $\alpha$-dilatation and its convergence hyperbolic interval  
$$
\HI (\Omega):= \left\{\alpha \in \mathbb{R}: Q(\alpha,\Omega)<\infty\right\}.
$$ do not depend on choice of a conformal mapping $\varphi:{\mathbb D}\to\Omega$ and can be reformulated in terms of the hyperbolic metrics \cite{BGU}. Namely
\begin{multline*}
\iint\limits _{\mathbb D}|\varphi'(z)|^{\alpha}~dxdy=
\iint\limits _{\mathbb D}\left(\frac{\lambda_{\mathbb D}(z)}{\lambda_{\Omega}(\varphi(z))}\right)^{\alpha}~dxdy \\
{}=
\iint\limits_{\Omega}\left|\left( \varphi^{-1}\right)'(w)\right|^{2-\alpha}dudv=
\iint\limits _{\Omega}\left(\frac{\lambda_{\mathbb D}\left( \varphi^{-1}(w)\right)}{\lambda_{\Omega}(w)}\right)^{2- \alpha}~dudv
\end{multline*}
where $\lambda_{\mathbb D}$ and $\lambda_{\Omega}$ are hyperbolic
metrics in $\mathbb D$ and $\Omega$ \cite{BM}. Let us recall that the hyperbolic metrics generated by $\lambda_{\mathbb D}\left( \varphi^{-1}(w)\right)$ are equivalent for different choice of conformal homeomorphism $\varphi:\Omega \to \mathbb D$ because any other conformal homeomorphism $\psi^{-1}:\Omega \to \mathbb D$ is a composition of $\varphi^{-1}$ and a M\"obius homeomorphism (that is an isometry of the hyperbolic metric).

%This dilatation measures a hyperbolic difference between the hyperbolic metric in the unit disc and $\Omega$.

\begin{remark} For any bounded simply connected domain $(-1.78,2]\subset  \HI(\alpha,\Omega)$ \cite{GU4, HSh}.  A more detailed discussion about the hyperbolic $\alpha$-dilatation and its convergence hyperbolic interval can be found in Appendix.
\end{remark}

In \cite{GU4} we proved that if a number $\alpha>2$ belongs to $\HI(\Omega)$ then $\Omega $ has finite geodesic diameter. By this reason we call domains that satisfy to a property $2<\alpha \in \HI(\Omega)$ as conformal $\alpha$-regular domains \cite{BGU}.

In this paper we obtain estimates of $Q(\alpha,\Omega)$ for a large class of conformal $\alpha$-regular domains (so-called quasidiscs) with the help of the exact inverse H\"older inequality for Jacobians of quasiconformal mappings. 
Using the estimates for $Q(\alpha,\Omega)$ we obtain estimates of constants for Sobolev-Poincar\'e inequalities and as a result we obtain lower estimates for first nontrivial eigenvalues of the Laplace operator with the Neumann boundary condition.

\subsection{Spectral Estimates of the Neumann-Laplacian} 

Let  $\Omega\subset\mathbb C$ be a bounded Lipschitz domain (an open connected set).  The Neumann eigenvalue problem
for the Laplace operator is:
\begin{equation*}
\label{Laplace}
\begin{cases}
-\operatorname{div}\left(\nabla u\right)=\mu u& \text{in}\,\,\,\,\,\Omega\\
 \frac{\partial u}{\partial n}=0& \text{on}\,\,\,\partial\Omega.
\end{cases}
\end{equation*}

The weak statement of this spectral problem is as follows: a function
$u$ solves the previous problem iff $u\in W^{1}_{2}(\Omega)$ and 
\[
\iint\limits _{\Omega} \nabla u(x,y)\cdot\nabla v(x,y)~dxdy=\mu \iint\limits _{\Omega}u(x,y)v(x,y)~dxdy
\]
for all $v\in W^{1}_{2}(\Omega)$. This statement is correct in any bounded domain $\Omega\subset\mathbb C$.

Let us  give a short historical review about eigenvalues estimates for the Neumann-Laplace operator. The classical upper estimate of the first nontrivial Neumann eigenvalue
$$
\mu_1(\Omega)\leq \mu_1(\Omega^{\ast})=\frac{p^2_{n/2}}{R^2_{\ast}}
$$
was proved by G.~Szeg\"o \cite{S54} for simply connected planar domains and by H.~F.~Weinberger  \cite{W56} for  domains in $\mathbb{R}^n$. In this inequality $p_{n/2}$ denotes the first positive zero of the function $(t^{1-n/2}J_{n/2}(t))'$, and $\Omega^{\ast}$ is an $n$-ball of the same $n$-volume as $\Omega$ with $R_{\ast}$ as its radius. In particular, if $n=2$, $p_1=j_{1,1}'\approx1.84118$ where $j_{1,1}'$ denotes the first positive zero of the derivative of the Bessel function $J_1$.
 
In 1961 G.~Polya  \cite{P60} obtained upper estimates for eigenvalues  in so-called plane-covering domains. Namely, for the first nontrivial eigenvalue $\mu_1(\Omega)$ it is:
$$
{\mu_1(\Omega)}\leq \frac{4\pi}{|\Omega|}.
$$

The upper estimate of the Laplace eigenvalues, with the help of different techniques, were intensively studied in recent decades, see, for example, \cite{A98,AB93,AL97,EP15,LM98}.

The lower estimates for the $\mu_1(\Omega)$ for convex domains were obtained in the classical work \cite{PW}. It was proved that if $\Omega$ is convex with diameter $d(\Omega)$ (see, also \cite{ENT, FNT, V12}), then
\begin{equation*}
\label{eq:PW}
\mu_1(\Omega)\geq \frac{\pi^2}{d(\Omega)^2}.
\end{equation*}

Unfortunately in non-convex domains $\mu_1(\Omega)$ can not be estimated in the terms of Euclidean diameters. It can be seen by considering a domain consisting of two identical squares connected by a thin corridor \cite{BCDL16}.
In \cite{GU16} was proved, on the basis of the geometric theory of composition operators on Sobolev spaces \cite{GGR95,U93} with applications to the (weighted) Poincar\'e-Sobolev inequalities \cite{GGu,GU}, that if $\Omega\subset\mathbb C$ be a conformal $\alpha$-regular domain, then the spectrum of Neumann-Laplace operator in $\Omega$ is discrete, can be written in the form
of a non-decreasing sequence
\[
0=\mu_0(\Omega)<\mu_{1}(\Omega)\leq\mu_{2}(\Omega)\leq...\leq\mu_{n}(\Omega)\leq...\,,
\]
and
\begin{equation}
\label{eq:est16}
{1}/{\mu_1(\Omega)} \leq \frac{4}{\sqrt[\alpha]{\pi^2}}\left(\frac{2\alpha-2}{\alpha-2}\right)^{\frac{2\alpha-2}{\alpha}}Q(\alpha,\Omega)^{2/\alpha}. 
\end{equation}

In the present article we suggest the approach which is based on the geometric theory of composition operators on Sobolev spaces and estimates of hyperbolic metric in quasidiscs. Recall that a domain $\Omega\subset\mathbb C$ is called a $K$-quasidisc if it is the image of the unit disc $\mathbb D$ under a $K$-quasiconformal homeomorphism of the complex plane $\mathbb C$ onto itself. Note that quasidiscs represent large class domains including fractal type domains like snowflakes. The Hausdorff dimensions of the quasidisc's boundary can be any number in $[1,2)$ \cite{HKN}.

Following \cite{Ahl66} a homeomorphism $\varphi:\Omega \rightarrow \Omega'$
between planar domains is called $K$-quasiconformal if it preserves
orientation, belongs to the Sobolev class $W_{2,\loc}^{1}(\Omega)$
and its directional derivatives $\partial_{\xi}$ satisfy the distortion inequality
$$
\max\limits_{\xi}|\partial_{\xi}\varphi|\leq K\min_{\xi}|\partial_{\xi}\varphi|\,\,\,
\text{a.e. in}\,\,\, \Omega \,.
$$

If $\Omega$ is a $K$-quasidisc, then a conformal mapping $\varphi: \mathbb D\to\Omega$ allows $K^2$-quasiconformal reflection. It is well known that Jacobians of quasiconformal mappings satisfy the weak inverse H\"older inequality \cite{BIw}. On the basis of the weak inverse H\"older inequality and the estimates of the constants in doubling conditions for measures generated by Jacobians of quasiconformal mappings (Proposition~\ref{L2.1}) we obtain

\vskip 0.3cm
{\bf Theorem A.}
\textit{Let $\Omega\subset\mathbb C$ be a $K$-quasidisc. Then the spectrum of Neumann-Laplace operator in $\Omega$ is discrete, can be written in the form
of a non-decreasing sequence
\[
0=\mu_0(\Omega)<\mu_{1}(\Omega)\leq\mu_{2}(\Omega)\leq...\leq\mu_{n}(\Omega)\leq...\,,
\]
and
\begin{equation}\label{eq:est}
\frac{1}{\mu_1(\Omega)} \leq   \frac{K^2C_\alpha^2}{\pi} \left(\frac{2 \alpha -2}{\alpha -2}\right)^{\frac{2 \alpha -2}{\alpha}}
\exp\left\{{\frac{K^2 \pi^2(2+ \pi^4)^2}{2\log3}}\right\}\cdot \big|\Omega\big|, 
\end{equation} 
for $2<\alpha<\frac{2K^2}{K^2-1}$, where
$$
C_\alpha=\frac{10^{6}}{[(\alpha -1)(1- \nu)]^{1/\alpha}}, \quad \nu = 10^{4 \alpha}\frac{\alpha -2}{\alpha -1}\left(24\pi^2K^2\right)^{\alpha}<1.
$$
}
\vskip 0.3cm

The main technical problem of this estimate is evaluations of the quasiconformality coefficient $K$ for quasidiscs. For this aim we use an equivalent description of quasidiscs in the terms of the Ahlfors's 3-point condition. This description of quasidiscs allows to obtain the estimates for the specific fractal type domains.

\vskip 0.3cm

{\bf Theorem C.}
\textit{Let $S_p\subset\mathbb C$, $1/4 \leq p < 1/2$, be the Rohde snowflake.
Then the spectrum of Neumann-Laplace operator in $S_p$ is discrete, 
can be written in the form of a non-decreasing sequence
\[
0=\mu_0(S_p)<\mu_{1}(S_p)\leq\mu_{2}(S_p)\leq...\leq\mu_{n}(S_p)\leq...\,,
\]
and
\begin{multline*}
\frac{1}{\mu_1(S_p)}
\leq \frac{C_\alpha^2 e^{4\left(1+e^{2\pi}(16/(1-2p))^5\right)^2}}{2^{40}\pi} \left(\frac{2 \alpha -2}{\alpha -2}\right)^{\frac{2 \alpha -2}{\alpha}}\\
{} \times \exp\left\{{\frac{\pi^2(2+ \pi^4)^2 e^{4\left(1+e^{2\pi}(16/(1-2p))^5\right)^2}}{2^{41}\log3}}\right\} \left|S_p\right|, 
\end{multline*}
for $2<\alpha<\frac{2K^2}{K^2-1}$, where
$$
C_\alpha=\frac{10^{6}}{[(\alpha -1)(1- \nu)]^{1/\alpha}}, \quad 
\nu = 10^{4 \alpha} \left(\frac{24\pi^2}{2^{40}}e^{4\left(1+e^{2\pi}C^5\right)^2}\right)^{\alpha}\frac{\alpha -2}{\alpha -1}<1.
$$
The quasiconformal coefficient $K$ of the Rohde snowflake $S_p$ satisfying the condition
$$
K< 2^{-10} \exp\left\{\left(1+e^{2 \pi}({16}/{(1-2p)})^5\right)^2\right\}. 
$$ 
}

\vskip 0.3cm

On the basis of Theorem A and Theorem C we can assert that eigenvalues of the Laplace operator depend on (quasi)conformal geometry of planar domains.

\section{Sobolev spaces and Poincar\'e-Sobolev inequalities}

Let $E\subset\mathbb C$ be a measurable set. For any $1\leq p<\infty$
we consider the Lebesgue space of locally integrable functions with the finite norm
\[
\|f\mid{L_{p}(E)}\|:=\left(\iint\limits _{E}|f(x,y)|^{p}~dxdy\right)^{1/p}<\infty.
\]

We define the Sobolev space $W^{1}_{p}(\Omega)$, $1\leq p<\infty$,
as a Banach space of locally integrable weakly differentiable functions
$f:\Omega\to\mathbb{R}$ equipped with the following norm: 
\[
\|f\mid W^{1}_{p}(\Omega)\|=\biggr(\iint\limits _{\Omega}|f(x,y)|^{p}\, dxdy\biggr)^{\frac{1}{p}}+\biggr(\iint\limits _{\Omega}|\nabla f(x,y)|^{p}\, dxdy\biggr)^{\frac{1}{p}}.
\]

We also define the homogeneous seminormed Sobolev space $L^1_p(\Omega)$, $1\leq p<\infty$,
of locally integrable weakly differentiable functions $f:\Omega\to\mathbb{R}$ equipped
with the following seminorm: 
\[
\|f\mid L^1_p(\Omega)\|=\biggr(\iint\limits _{\Omega}|\nabla f(x,y)|^{p}\, dxdy\biggr)^{\frac{1}{p}}.
\]

Recall that the embedding operator $i:L^1_p(\Omega)\to L_{1,\loc}(\Omega)$
is bounded.

If $f\in W^1_p(\mathbb D)$, $1\leq p<\infty$, then for $0\leq \kappa=1/p-1/q<1/2$ the Poincar\'e-Sobolev inequality 
$$
\left(\iint\limits_{\mathbb D}|f(z)-f_{\mathbb D}|^q~dxdy\right)^{\frac{1}{q}}\leq 
B_{q,p}(\mathbb D)\left(\iint\limits_{\mathbb D}|\nabla f(z)|^p~dxdy\right)^{\frac{1}{p}}
$$
holds (see, for example, \cite{GT,GU16}) with the constant 
$$
B_{q,p}(\mathbb D)\leq \frac{2}{\pi^{\kappa}}\left(\frac{1-\kappa}{1/2-\kappa}\right)^{1-\kappa}.
$$ 

This estimate is not applicable in the critical case $p=1$ and $q=2$.
Now we obtain the upper estimate of the Poincar\'e constant in the Poincar\'e-Sobolev inequality for the Sobolev space $W^1_1(\mathbb D)$ in this critical case. We use the following Gagliardo inequality 
for functions with compact support \cite{FF, M}:
\begin{equation}\label{GaglIn}
\left(\iint\limits_{\Omega} \big|f(z)\big|^2~dxdy \right)^{\frac{1}{2}} \leq
\frac{1}{2\sqrt{\pi}} \iint\limits_{\Omega} |\nabla f(z)|~dxdy,
\end{equation}
where $\Omega \subset \mathbb{C}$ be a bounded Lipschitz domain.

\begin{theorem} 
Let $f \in W_1^1(\Omega)$. Then for any $r>0$ and any $z_0\in\Omega: \dist(z_0,\partial\Omega)>2 r$, the following inequality 
\begin{equation}\label{PS21}
\left(\iint\limits_{D(z_0,r)} |f(z)-f_{D(z_0,r)}|^2~dxdy\right)^{\frac{1}{2}} \leq \frac{3\sqrt{\pi^3}}{4} \iint\limits_{D(z_0,r)} |\nabla f(z)|~dxdy
\end{equation}
holds.
\end{theorem}

\begin{proof} 
We prove this inequality in the case $f_{D(z_0,r)}=0$. In this case the inequity (\ref{PS21}) can be rewritten as
\begin{equation}\label{PS210}
\left(\iint\limits_{D(z_0,r)} |f(z)|^2~dxdy\right)^{\frac{1}{2}} \leq \frac{3\sqrt{\pi^3}}{4} \iint\limits_{D(z_0,r)} |\nabla f(z)|~dxdy.
\end{equation}
The inequality (\ref{PS210}) is invariant under translations and similarities and it is sufficient to prove one in for the disc $D(0,1)$. 

Denote by $D(0,\delta)$ an open disc of the radius $\delta>1$. Choose the cut function $\eta$ in the form $\eta= \eta(|z|)$ such that
$$
\eta(z)=
\begin{cases}1, \,\,\text{if}\,\, |z|<1,\\
0, \,\,\text{if}\,\, |z|>\delta
\end{cases}
$$
and linear for $z \in R_{\delta}$, where $R_{\delta}:=D(0,\delta) \setminus \overline{D(0,1)}$. Then
\[ \left|\nabla \eta(z)\right| = \begin{cases}
\frac{1}{\delta -1} & \text{if $z \in R_{\delta}$,}\\
0 & \text{otherwise.}
\end{cases} \]

Let $\widetilde{f}:R_{\delta}\to \mathbb R$ be the extension function is defined by the rule
$$
\widetilde{f}(z)=f\left(w(z)\right)\eta(z),\,\,w\in R_{\delta},
$$ 
where $w(z)=1/\overline{z}$ is the inversion in the unit circle $S(0,1)$.

Define the extension operator on Sobolev spaces
$$
E:L_1^1(D(0,1)) \to L_1^1(D(0,\delta))
$$ 
by the formula 
\[ (Ef)(z) = \begin{cases}
f(z) & \text{if $z \in D(0,1)$,} \\
\widetilde{f}(z) & \text{if $z \in R_{\delta}$.}
\end{cases} \]

In order to estimate the norm $\|E\|$ of the extension operator $E:L_1^1(D(0,1)) \to L_1^1(D(0,\delta))$ we have
$$
\|E\mid L_1^1(\mathbb D(0,\delta))\|= \iint\limits_{D(0,1)} |\nabla f(z)|~dxdy + \iint\limits_{R_{\delta}} |\nabla \widetilde{f}(z)|~dxdy. 
$$

To estimate the second integral in the right side of this equality by elementary calculations we have
\begin{multline}
|\nabla \widetilde{f}(z)|=|\nabla \left(f(w(z))\eta(z)\right)|=
|\nabla \left(f(w(z))\right)\cdot \eta(z)+f(w(z))\cdot\nabla \eta(z)|\\
\leq
\left|\nabla f(w(z))\right| +\frac{1}{\delta -1}\left|f(w(z))\right|.
\end{multline}
Hence
$$
\iint\limits_{R_{\delta}} |\nabla \widetilde{f}(w(z))|~dxdz \\
{} \leq \iint\limits_{R_{\delta}} \left|\nabla f(w(z))\right|~dxdy +
\frac{1}{\delta -1}\iint\limits_{R_{\delta}} \left|f(w(z))\right|~dxdy.
$$

First consider the integral
\begin{multline*}
\iint\limits_{R_{\delta}} |\nabla f(w(z))|~dxdy=\iint\limits_{R_{\delta}}|\nabla f|(w(z))|w'(z)|dxdy=\\
=
\iint\limits_{R_{\delta}}|\nabla f|(w(z))|w'(z)||J(z,w)||J(z,w)|^{-1}dxdy \\
\leq \sup\limits_{z \in R_{\delta}} \frac{|w'(z)|}{|J(z,w)|} \iint\limits_{R_{\delta}} |\nabla f|(w(z))|J(z,w)|~dxdy \\
{} = \sup\limits_{z \in R_{\delta}} \frac{|w'(z)|}{|J(z,w)|} \iint\limits_{D_{\delta}} |\nabla f|(w)~dudv,
\end{multline*}
where $w:R_{\delta} \to D_{\delta}$, $w(z)=1/\overline{z}$ and $D_{\delta}=\{w\in\mathbb C: 1/\delta<|z|<1\}$. We calculate the norm of the derivative of mapping $w$ by the formula 
$$
|w'(z)|=|w_z|+|w_{\overline{z}}|
$$
and the Jacobian of mapping $w$ by the formula 
$$
J(z,w)=|w_z|^2-|w_{\overline{z}}|^2.
$$
Here
$$
w_z=\frac{1}{2}\left(\frac{\partial w}{\partial x}-i\frac{\partial w}{\partial y}\right) \quad \text{and} \quad 
w_{\overline{z}}=\frac{1}{2}\left(\frac{\partial w}{\partial x}+i\frac{\partial w}{\partial y}\right).
$$ 

By elementary calculations
$$
w_z=0 \quad \text{and} \quad w_{\overline{z}}=-\frac{1}{\overline{z}^2}.
$$ 
Hence
$$
|w'(z)|=\frac{1}{|\overline{z}|^2} \quad \text{and} \quad |J(z,w)|=\frac{1}{|\overline{z}|^4}.
$$
Finally we get
$$
\iint\limits_{R_{\delta}} |\nabla f(w(z))|~dxdy \leq \sup\limits_{z \in R_{\delta}} |\overline{z}|^2 \iint\limits_{D_{\delta}} |\nabla f|(w)~dudv
\leq \delta^2 \iint\limits_{D(0,1)} |\nabla f|(w)~dudv.
$$

In order to estimate the integral
$$
\iint\limits_{R_{\delta}} \left|f\left(w(z)\right)\right|~dxdy
$$
we will use the change of variable formula. 

We obtain
\begin{multline*}
\iint\limits_{R_{\delta}} |f\left(w(z)\right)|~dxdy=\iint\limits_{R_{\delta}} |f\left(w(z)\right)||J(z,w)||J(z,w)|^{-1}~dxdy \\
\leq \sup\limits_{z \in R_{\delta}} \frac{1}{|J(z,w)|} \iint\limits_{R_{\delta}} |f\left(w(z)\right)||J(z,w)|~dxdy 
\\
{} = \sup\limits_{z \in R_{\delta}} \frac{1}{|J(z,w)|} \iint\limits_{D_{\delta}} |f(w)|~dudv \leq \delta^4 \iint\limits_{D(0,1)} |f(w)|~dudv. 
\end{multline*} 

Using the following Poincar\'e--Sobolev inequality \cite{AD}
$$
\iint\limits_{\Omega} |f(z)|~dxdy \leq \frac{d}{2} \iint\limits_{\Omega} |\nabla f(z)|~dxdy
$$
where $\Omega$ be a convex domain with diameter $d$ and $f \in W_1^1(\Omega)$, finally we obtain
$$
\iint\limits_{R_{\delta}} |f\left(w(z)\right)|~dxdy\leq \delta^4 \iint\limits_{D(0,1)} |\nabla f(w)|~dudv.
$$
Thus
$$
\iint\limits_{R_{\delta}} |\nabla \widetilde{f}(z)|~dxdy 
\leq \delta^2 \iint\limits_{D(0,1)} |\nabla f(w)|~dudv +
\frac{\delta^4}{\delta -1}\iint\limits_{D(0,1)} |\nabla f(w)|~dudv,
$$
and consequently
\begin{multline*}
\|E(f)\mid L_1^1( D(0,\delta))\| \leq \left(1+ \delta^2 + \frac{\delta^4}{\delta -1}\right) \iint\limits_{D(0,1)} |\nabla f(z)|~dxdy\\
=
\left(1+ \delta^2 + \frac{\delta^4}{\delta -1}\right)\|f\mid L_1^1( D(0,1))\|.
\end{multline*}

Now, according to inequality \eqref{GaglIn} we obtain
\begin{multline*}
\left(\iint\limits_{D(0,1)} |f(z)|^2~dxdy\right)^{\frac{1}{2}} \\
{} \leq \left(\iint\limits_{D(0,\delta)} |Ef(z)|^2~dxdy\right)^{\frac{1}{2}}\leq \frac{1}{2\sqrt{\pi}} \iint\limits_{D(0,\delta)} |\nabla Ef(z)|~dxdy\\
\leq \frac{\delta^4+\delta^3-\delta^2+\delta-1}{2\sqrt{\pi}(\delta-1)} \iint\limits_{D(0,1)} |\nabla f(z)|~dxdy\leq 
C(\delta) \iint\limits_{D(0,1)} |\nabla f(z)|~dxdy.
\end{multline*}

Taking the optimal $\delta \approx 5/4$ we have
$$
\left(\iint\limits_{D(0,1)} |f(z)|^2~dxdy\right)^{\frac{1}{2}}\leq
\frac{3\sqrt{\pi^3}}{4}\iint\limits_{D(0,1} |\nabla f(z)|~dxdy.
$$
\end{proof}

\section{The Doubling Condition and the H\"older Inequality}

Here we recall necessary facts about conformal capacity.
A well-ordered triple $(F_{0},F_{1};\Omega)$ of nonempty sets, where
$\Omega$ is an open set in $\mathbb C$, and $F_{0}$,
$F_{1}$ are closed subsets of $\Omega$, is called a condenser on
the complex plane $\mathbb C$.

The value 
\[
\cp(E)=\cp(F_{0},F_{1};\Omega)=\inf\iint\limits _{\Omega}|\nabla v|^{2}dxdy,
\]
where the infimum is taken over all Lipschitz nonnegative functions
$v:\overline{\Omega}\to\mathbb R$, such that $v=0$ on
$F_{0}$, and $v=1$ on $F_{1}$, is called conformal capacity of the condenser
$E=(F_{0},F_{1};\Omega)$. If the set of admissible functions is empty, then $\cp(F_{0},F_{1};\Omega)=\infty$.
For finite values of capacity $0<\cp(F_{0},F_{1};\Omega)<+\infty$
there exists a unique continuous weakly differentiable function $u_{0}$
(an extremal function) such that: 
\[
\cp(F_{0},F_{1};\Omega)=\iint\limits _{\Omega}|\nabla u_{0}|^{2}dxdy.
\]

Quasiconformal mappings can be characterized in capacitary terms (see, for example \cite{GR}). Namely,
a homeomorphism $\varphi:\Omega\rightarrow\Omega'$, $\Omega,\Omega'\subset\mathbb C$ is $K$-quasiconformal, 
if and only if 
\[
K^{-1}\cp(F_{0},F_{1};\Omega)\leq\cp(\varphi\left(F_{0}\right),\varphi\left(F_{1}\right);\Omega')\leq K\:\cp(F_{0},F_{1};\Omega)
\]
for any condenser $E=(F_{0},F_{1};\Omega)$.

We will need the following estimate of the conformal capacity (see, for example \cite{GR, Vu}).

Denote $D(z_0,r)$ an open disc with center $z_0$ of radius $r$ and $\overline{D(z_0,r)}$ its closure.

\begin{lemma}\cite{GR}
\label{L1.1}
Let $R >r>0$. Then 
$\cp\left(\overline{D(z_0,r)}, \mathbb C \setminus D(z_0,R);\mathbb C\right)=$ $2 \pi (\log R/r)^{-1}$. 
\end{lemma}

Consider capacity estimates for Teichm\"uller type condensers in $\overline{\mathbb C}$
(see, for example \cite{Vu}, Lemma 5.32).

\begin{lemma}\label{L1.2} 
Fix $0<r<R$. Let $F_0$ and $F_1$ be continuums in $\mathbb C$ such that
\[
F_0 \cap S(0,\rho) \ne \emptyset\,\,\text{and}\,\, F_1 \cap S(0,\rho)\ne \emptyset
\]
for all $r< \rho <R$, where $S(0,\rho)$ is the circle of radius $\rho$.
Then 
$$
\cp(F_0,F_1;D(0,R) \setminus \overline{D(0,r)}) \geq \frac{2}{\pi} \log \frac{R}{r}.
$$
\end{lemma}

Denote by $R_T(t)=\overline{\mathbb C} \setminus \left\{[-1,0] \cup [t,\infty]\right\}$, $t>1$, the Teichm\"uller ring in 
$\overline{\mathbb C}$.

\begin{lemma}\label{L1.3} 
Let $R_T(t)$ be the Teichm\"uller ring in $\overline{\mathbb C}$. Then
\[
\cp R_T(t)=\frac{2 \pi}{\log \Phi (t)},
\]
where $\Phi$ satisfies the conditions
\[
t+1 \leq \Phi(t) < 32t,\,\,t>1.
\]
\end{lemma}

Using this capacity estimates we obtain estimates for a constant in the inverse H\"older inequality. We start from the weak H\"older inequality.

\begin{lemma}\label{L4.1}
Let $\varphi : \Omega \to \Omega'$ be a $K$-quasiconformal mapping of planar domains $\Omega,\Omega'\subset\mathbb C$. 
Then for every disc $D(z_0,r)$ such that $D(z_0,2r)\subset \Omega$ the weak inverse H\"older inequality for the first generalized derivatives of $\varphi$ 
\begin{equation}\label{WIHI} 
\left(\frac{1}{|D(z_0,r)|} \iint\limits_{D(z_0,r)} |\varphi'(z)|^2~dxdy\right)^{\frac{1}{2}} \leq
\frac{24\pi^2K}{|D(z_0,2r)|}\iint\limits_{D(z_0,2r)} |\varphi'(z)|~dxdy 
\end{equation}
holds.
\end{lemma}

\begin{proof} Derivatives of $K$-quasiconformal mappings satisfy the following inequality (\cite{BIw}, pp 274)
\begin{multline*}
\left(\frac{1}{|D(z_0,r)|} \iint\limits_{D(z_0,r)} |\varphi'(z)|^2~dxdy\right)^{\frac{1}{2}} \\
{} \leq
\frac{4K}{r} \left(\frac{1}{|D(z_0,2r)|}\iint\limits_{D(z_0,2r)} \bigg|\varphi(z)-\frac{1}{|D(z_0,2r)|}\iint\limits_{D(z_0,2r)} \varphi(z)\bigg|^2~dxdy \right)^{\frac{1}{2}}.
\end{multline*}
Applying to the right side of this inequality the Poincar\'e-Sobolev inequality~\eqref{PS21},
%\begin{multline*}
%\left(\iint\limits_{D(z_0,2r)} \bigg|f(z)-\frac{1}{|D(z_0,2r)|}\iint\limits_{D(z_0,2r)} f(z)\bigg|^2~dxdy \right)^{\frac{1}{2}} \\
%{} \leq \frac{3\sqrt{\pi^3}}{4} \iint\limits_{D(z_0,2r)} |\nabla f(z)|~dxdy,
%\end{multline*}
we obtain the required inequality~(\ref{WIHI}).
\end{proof}

Directly by Lemma~\ref{L4.1} and Theorem 4.2 from~\cite{BIw} we get the following weak H\"older inequality:

\begin{theorem}\label{T2.1}
Let $\varphi : \Omega \to \Omega'$ be a $K$-quasiconformal mapping of planar domains $\Omega,\Omega'\subset\mathbb C$. 
Then for every disc $D(z_0,r)$ such that $D(z_0,2r)\subset \Omega$ and for some $\sigma >2$  the inequality
\begin{multline}\label{EsDer} 
\left(\frac{1}{|D(z_0,r)|} \iint\limits_{D(z_0,r)} |\varphi'(z)|^{\sigma}~dxdy\right)^{\frac{1}{\sigma}} \\
{} \leq C_{\sigma}\left(\frac{1}{|D(z_0,2r)|}\iint\limits_{D(z_0,2r)} |\varphi'(z)|^2~dxdy\right)^{\frac{1}{2}} 
\end{multline}
holds, where 
$$
C_{\sigma}=\frac{10^6}{[(\sigma -1)(1- \nu)]^{1/\sigma}}, \quad \nu = 10^{4 \sigma}\frac{\sigma -2}{\sigma -1}\left(24\pi^2K\right)^{\sigma}<1.
$$
\end{theorem}

Given Theorem~\ref{T2.1} for further estimates of the left side of the inequality (\ref{EsDer}) we use the doubling property of measures generated by Jacobians of quasiconformal mappings (see, for example, \cite{HK95}). Recall that a Borel measure $\mu$ on a set $\Omega$ is doubling if there exist a constant $C_{\mu} \geq 1$ so that the inequality
\[
\mu(D(z_0,2r)) \leq C_{\mu} \cdot \mu(D(z_0,r))<\infty
\]
hold for all discs $D(z_0,r)$ in $\Omega$.

In the following proposition we give an estimate of the constant $C_{\mu}$ in the measure doubling condition.

\begin{proposition}\label{L2.1}
Let $\varphi: \mathbb C \to \mathbb C$ be a $K$-quasiconformal mapping.
Then for any $z_0 \in \mathbb C$ and any $r>0$ we have
\begin{equation}\label{EsDC} 
\iint\limits_{D(z_0,2r)} |J(z,\varphi)|~dxdy \leq \exp\left\{{\frac{K \pi^2(2+ \pi^4)^2}{2\log3}}\right\}
\iint\limits_{D(z_0,r)} |J(z,\varphi)|~dxdy. 
\end{equation}
\end{proposition}

\begin{proof} Because the inequality (\ref{EsDC}) is invariant under translations and similarities we can suppose that 
$z_0=0$ and radius $r=1$. It is sufficient to show that
\begin{equation}
\label{double}
|\varphi(D(0,2))|\leq \exp\left\{{\frac{K \pi^2(2+ \pi^4)^2}{2\log3}}\right\}|\varphi(D(0,1))|.
\end{equation}

Since $\varphi$ is a quasiconformal mapping, $\varphi(0)$ is an interior point of the open
connected set $U_0= \varphi(D(0,1))$. Denote by $\lambda_0=\dist(\varphi(0),\partial U_0)$ and by 
$\lambda_1=\max\limits_{w\in \partial U_1}|\varphi(0)-w|$, $U_1= \varphi(D(0,2))$. 

For the proof of the inequality \ref{double} we estimate the ratio
$$
\frac{|\varphi(D(0,2))|}{|\varphi(D(0,1))|}\leq \frac{\pi \lambda_1^2}{\pi \lambda_0^2}=\frac{\lambda_1^2}{\lambda_0^2},
$$
because the disc $D(\varphi(0),\lambda_1)$ contains $\varphi(D(0,2))$ and the disc $D(\varphi(0),\lambda_0)$ is contained in $\varphi(D(0,1))$. For this aim we use the capacity estimates and quasi-invariance of conformal capacity under quasiconformal mappings and consider contimuums $F_0$ and $F_1$ where $F_0$ be a line segment of length $\lambda_0$ joining $\varphi(0)$ to $w_0\in\partial U_0$ and $F_1$ be a continuum in $\overline{\mathbb C}= \mathbb C \cup \{\infty\}$ joining a point $w_1$ in $\partial D(\varphi(0),\lambda_1) \cap U_1$ 
to $\infty$ in $\overline{\mathbb C}\setminus D(\varphi(0),\lambda_1)$. 

Now we consider pre-images of continuums $\varphi^{-1}(F_0)$ and $\varphi^{-1}(F_1)$ 
in order to obtain a lower estimate of the capacity of the condenser $(\varphi^{-1}(F_0), \varphi^{-1}(F_1); D(0,\pi))$. 
Because capacity is invariant under rotations, without loss of generality we can suppose that $z_0=\varphi^{-1}(w_0)=(1,0)$.
Let $z_1=\varphi^{-1}(w_1)=(2,\theta_0)$, where  $(1,0)$ and $(2,\theta_0)$ are the polar coordinates.

Define the bi-Lipschitz mapping $\psi:\mathbb C\to\mathbb C$ which maps $z_0$ to $z_0$ and $z_1$ to $(2,0)$ by the rule
$$
\psi(\rho,\theta)=(\rho,\theta - (\rho-1)\theta_0), 
$$
where $(\rho, \theta)$ are polar coordinates on the complex plane $\mathbb C$. 

Now we estimate the Lipschitz coefficient $L$ for the mapping $\psi$. For this aim we represent $\psi$ as 
$$
\psi(x,y)=(\sqrt{x^2+y^2} \cos t, \sqrt{x^2+y^2} \sin t), 
$$
where 
$$
t=\arctan \frac{y}{x} - (\sqrt{x^2+y^2}-1)\theta_0, \quad 0 \leq \theta_0 \leq \pi,
$$
and we use the following expressions
$$
L=\|D \psi (x,y)\|_{o} \leq \frac{\sqrt{2}}{2} \|D \psi (x,y)\|_{e}.
$$
Here $D \psi (x,y)$ denotes the Jacobian matrix of the mapping $\psi = \psi (x,y)$, 
$\|\cdot\|_{o}$ and $\|\cdot\|_{e}$ denote the operator norm and the Euclidean norm respectively.

The Jacobian matrix has the form
$$
D \psi (x,y)=
\begin{pmatrix} 
\frac{x \cos t +y \sin t}{\sqrt{x^2+y^2}} + x \theta_0 \sin t & \frac{y \cos t -x \sin t}{\sqrt{x^2+y^2}} + y \theta_0 \sin t \\ 
- \frac{y \cos t -x \sin t}{\sqrt{x^2+y^2}}-x \theta_0 \cot t  & \frac{x \cos t +y \sin t}{\sqrt{x^2+y^2}} -y \theta_0 \cot t
\end{pmatrix}.
$$

A straightforward calculation yields
$$
 \|D \psi (x,y)\|_{e} = \sqrt{2+ (x^2+y^2) \theta_0^2}.
$$
Hence
$$
L \leq \frac{\sqrt{2}}{2} \|D \psi (x,y)\|_{e} = \sqrt{1+ \frac{(x^2+y^2) \theta_0^2}{2}}
\leq \sqrt{1+\frac{\pi^4}{2}}.
$$

Now we consider the ring $R=D(\frac{3}{2},\frac{3}{2}) \setminus D(\frac{3}{2},\frac{1}{2})$. Then by Lemma~\ref{L1.2} using the monotonicity of the capacity and quasi-invariance under bi-Lipschtz mappings we obtain
\begin{equation}\label{InL}
\cp(\varphi^{-1}(F_0), \varphi^{-1}(F_1); D(0,\pi)) \geq \frac{1}{L^4} \cp(E_0,E_1;R) \geq \frac{8 \log3}{\pi (2+ \pi^4)^2}, 
\end{equation}
where $E_0= \psi(\varphi^{-1}(F_0)) \cap R$, $E_1= \psi(\varphi^{-1}(F_1)) \cap R$.

On the other hand, taking into account the capacitary definition of the quasiconformal mapping and Lemma~\ref{L1.1} we have
\begin{multline}\label{InR}
\cp(\varphi^{-1}(F_0), \varphi^{-1}(F_1); D(0,\pi)) \leq K \cp(F_0,F_1;\mathbb C) \\
{} \leq K \cp(D(\varphi(0),\lambda_0), \mathbb C \setminus D(\varphi(0),\lambda_1);\mathbb C)
=\frac{2 \pi K}{\log \frac{\lambda_1}{\lambda_0}}.
\end{multline}

Combining inequalities~\eqref{InL} and~\eqref{InR}, we obtain
\[
\frac{8 \log3}{\pi (2+ \pi^4)^2} \leq \frac{2 \pi K}{\log \frac{\lambda_1}{\lambda_0}}.
\]
By elementary calculations
\[
\lambda_1 \leq \exp\left\{{\frac{K \pi^2(2+ \pi^4)^2}{4\log3}}\right\} \lambda_0.
\]

In order to obtain the required inequality it is necessary to perform the following straightforward calculations
\begin{multline*}
|\varphi(D(0,2))| \leq \pi \lambda_1^2 \\
{} \leq \pi \exp\left\{{\frac{K \pi^2(2+ \pi^4)^2}{2\log3}}\right\} \lambda_0^2 \leq \exp\left\{{\frac{K \pi^2(2+ \pi^4)^2}{2\log3}}\right\} |\varphi(D(0,1))|.
\end{multline*}
Thus
\[
\iint\limits_{D(0,2)} |J(z,\varphi)|~dxdy \leq \exp\left\{{\frac{K \pi^2(2+ \pi^4)^2}{2\log3}}\right\}
\iint\limits_{D(0,1)} |J(z,\varphi)|~dxdy.
\]
\end{proof}

For any planar $K$-quasiconformal homeomorphism $\varphi:\Omega\rightarrow \Omega'$
the following sharp result is known: $J(z,\varphi)\in L_{p,\loc}(\Omega)$
for any $1 \leq p<\frac{K}{K-1}$ (\cite{A, G1}).
 
\begin{proposition}
\label{prop:confQuasidisc} For any conformal homeomorphism $\varphi:\mathbb{D}\to\Omega$
of the unit disc $\mathbb{D}$ onto a $K$-quasidisc $\Omega$ derivatives $\varphi'\in L_p(\mathbb{D})$ for any 
$1\le p<\frac{2K^{2}}{K^2-1} \subset \HI(\Omega) $.
\end{proposition}

\begin{proof} Any conformal homeomorphism $\varphi:\mathbb{D}\to\Omega$ can be extended to a $K^2$ quasiconformal homeomorphism  
$\widetilde{\varphi}$ of the whole plane to the whole plane by reflection. Hence 
$\widetilde{\varphi}'$ belongs to the class $L_{p,\loc}(\mathbb C)$ for any $1\le p<\frac{2K^2}{K^2-1}$ (\cite{A, G1}).
Therefore $\varphi'$ belongs to the class $L_{p}(\mathbb D)$.
\end{proof}

On the base of the weak H\"older inequality and the doubling condition we obtain integral estimates of complex derivatives of conformal mapping $\varphi:\mathbb D\to\Omega$ in the unit disc onto a $K$-quasidisc $\Omega$:

\vskip 0.3cm
{\bf Theorem B.}
\textit{Let $\Omega\subset\mathbb R^2$ be a $K$-quasidisc and $\varphi:\mathbb D\to\Omega$ be a conformal mapping. Suppose that  $2<\alpha<\frac{2K^2}{K^2-1}$.
Then 
\begin{equation}\label{Ineq_2}
\left(\iint\limits_{D(0,1)} |J(z,\varphi)|^{\frac{\alpha}{2}}~dxdy \right)^{\frac{2}{\alpha}}
\leq \frac{C_\alpha^2 K^2 \pi^{\frac{2}{\alpha}-1}}{4}
\exp\left\{{\frac{K^2 \pi^2(2+ \pi^4)^2}{2\log3}}\right\}\cdot |\Omega|.
\end{equation}
where
$$
C_\alpha=\frac{10^{6}}{[(\alpha -1)(1- \nu)]^{1/\alpha}}, \quad \nu = 10^{4 \alpha}\frac{\alpha -2}{\alpha -1}(24\pi^2K^2)^{\alpha}<1.
$$
}
\vskip 0.3cm

\begin{proof}
Since $\varphi:\mathbb D\to\Omega$ is a conformal mapping, then
by the inequality~\eqref{EsDer} using the equality $J(z,\varphi)=|\varphi'(z)|^2$ we have
\begin{multline*}
\left(\iint\limits_{D(0,1)} |J(z,\varphi)|^{\frac{\alpha}{2}}~dxdy \right)^{\frac{2}{\alpha}} =
|D(0,1)|^{\frac{2}{\alpha}}\left(\frac{1}{|D(0,1)|}\iint\limits_{D(0,1)} |\varphi'(z)|^{\alpha}~dxdy \right)^{\frac{1}{\alpha}\cdot 2} \\
{} \leq C_\alpha^2|D(0,1)|^{\frac{2}{\alpha}}\left(\frac{1}{|D(0,2)|}\iint\limits_{D(0,2)} |\varphi'(z)|^{2}~dxdy \right) =
\frac{C_\alpha^2 \pi^{\frac{2}{\alpha}-1}}{4}\iint\limits_{D(0,2)} |\varphi'(z)|^{2}~dxdy.
\end{multline*}

In the disc $D(0,2)$ an extension of the conformal mapping $\varphi$ is a $K^2$-quasiconformal homeomorphism. Hence
taking into account the inequality 
$$
|\varphi'(z)|^{2} \leq K^2 |J(z,\varphi)|\,\,\text{ for almost all}\,\, z\in D(0,2)
$$ 
by the inequality~\eqref{EsDC} we get
\begin{multline*}
\frac{C_\alpha^2 \pi^{\frac{2}{\alpha}-1}}{4}\iint\limits_{D(0,2)} |\varphi'(z)|^{2}~dxdy \leq
\frac{C_\alpha^2 K^2 \pi^{\frac{2}{\alpha}-1}}{4}\iint\limits_{D(0,2)} |J(z,\varphi)|~dxdy \\
{} \leq \frac{C_\alpha^2 K^2 \pi^{\frac{2}{\alpha}-1}}{4}
\exp\left\{{\frac{K^2 \pi^2(2+ \pi^4)^2}{2\log3}}\right\}\iint\limits_{D(0,1)} |J(z,\varphi)|~dxdy.
\end{multline*}

Thus, considering that
\[
\iint\limits_{D(0,1)} |J(z,\varphi)|~dxdy=|\Omega|
\]
we have
\begin{equation*}
%\label{Ineq_2}
\left(\iint\limits_{D(0,1)} |J(z,\varphi)|^{\frac{\alpha}{2}}~dxdy \right)^{\frac{2}{\alpha}}
\leq \frac{C_\alpha^2 K^2 \pi^{\frac{2}{\alpha}-1}}{4}
\exp\left\{{\frac{K^2 \pi^2(2+ \pi^4)^2}{2\log3}}\right\}\cdot |\Omega|.
\end{equation*}

\end{proof}

\section{Estimates of the hyperbolic $\alpha$-dilatation and the first non-trivial eigenvalue of the Neumann--Laplace operator}

In the work \cite{GU16} was obtained the following estimates of the first non-trivial eigenvalue of the Neumann-Laplace operator in quasidiscs:

\begin{proposition} \label{quasidisc} Suppose a conformal homeomorphism $\varphi:\mathbb{D}\to\Omega$ maps
the unit disc $\mathbb{D}$ onto a $K$-quasidisc $\Omega$. Then
\[
{1}/{\mu_1(\Omega)} \leq \frac{4}{\sqrt[\alpha]{\pi^2}}\left(\frac{2\alpha-2}{\alpha-2}\right)^{\frac{2\alpha-2}{\alpha}}Q(\alpha,\Omega)^{2/\alpha}
\nonumber
\]
for any $2<\alpha<\frac{2K^2}{K^2-1}.$

\end{proposition}

By Proposition \ref{quasidisc} and Theorem B we obtain the lower estimates of the first non-trivial eigenvalue of the Neumann--Laplace operator in $K$-quasidisc in terms of quasiconformal geometry of domains:  

\vskip 0.3cm
{\bf Theorem A.}
\textit{Let $\Omega\subset\mathbb C$ be a $K$-quasidisc. Then the spectrum of Neumann-Laplace operator in $\Omega$ is discrete, can be written in the form
of a non-decreasing sequence
\[
0=\mu_0(\Omega)<\mu_{1}(\Omega)\leq\mu_{2}(\Omega)\leq...\leq\mu_{n}(\Omega)\leq...\,,
\]
and
\begin{equation}\label{InH}
\frac{1}{\mu_1(\Omega)} \leq   \frac{K^2C_\alpha^2}{\pi} \left(\frac{2 \alpha -2}{\alpha -2}\right)^{\frac{2 \alpha -2}{\alpha}}
\exp\left\{{\frac{K^2 \pi^2(2+ \pi^4)^2}{2\log3}}\right\}\cdot \big|\Omega\big|, 
\end{equation} 
for $2<\alpha<\frac{2K^2}{K^2-1}$, where
$$
C_\alpha=\frac{10^{6}}{[(\alpha -1)(1- \nu)]^{1/\alpha}}, \quad \nu = 10^{4 \alpha}\frac{\alpha -2}{\alpha -1}(24\pi^2K^2)^{\alpha}<1.
$$
}
\vskip 0.3cm

\begin{proof}
By Proposition \ref{quasidisc} for any $2<\alpha<\frac{2K^2}{K^2-1}$ we have
\begin{equation}\label{Ineq_1}
{1}/{\mu_1(\Omega)} \leq \frac{4}{\sqrt[\alpha]{\pi^2}}\left(\frac{2\alpha-2}{\alpha-2}\right)^{\frac{2\alpha-2}{\alpha}}Q(\alpha,\Omega)^{2/\alpha}.
\end{equation}

By Theorem B
\begin{multline}
\label{Ineq_3}
Q\left(\alpha, \Omega\right)^{\frac{2}{\alpha}} 
{} :=\left(\iint\limits_{D(0,1)} |J(z,\varphi)|^{\frac{\alpha}{2}}~dxdy \right)^{\frac{2}{\alpha}}\\
\leq \frac{C_\alpha^2 K^2 \pi^{\frac{2}{\alpha}-1}}{4}
\exp\left\{{\frac{K^2 \pi^2(2+ \pi^4)^2}{2\log3}}\right\}\cdot |\Omega|.
\end{multline}

Combining inequalities~\eqref{Ineq_1} and~\eqref{Ineq_3} and performing a straightforward calculations we obtain the required inequality.
\end{proof}

As an application of Theorem~A, we obtain the lower estimates of the first non-trivial eigenvalue on the 
Neumann eigenvalue problem for the Laplace operator in the star-shaped and spiral-shaped domains.

A simply connected domain $\Omega^*$ is $\beta$-star-shaped (with respect to $z_0=0$)
if the function $\varphi(z)$, $\varphi(0)=0$, conformally maps a unit disc $\mathbb{D}$ onto $\Omega^*$ and satisfies the condition \cite{FKZ}:
\[
\left|\arg \frac{z \varphi^{\prime}(z)}{\varphi(z)} \right| \leq \beta \pi/2, \quad 0 \leq \beta <1, \quad |z|<1.
\]

A simply connected domain $\Omega_s$ is $\beta$-spiral-shaped (with respect to $z_0=0$)
if the function $\varphi(z)$, $\varphi(0)=0$, conformally maps a unit disc $\mathbb{D}$ onto $\Omega_s$ and satisfies the condition \cite{S87, Sug}:
\[
\left|\arg e^{i \delta} \frac{z \varphi^{\prime}(z)}{\varphi(z)} \right| \leq \beta \pi/2, \quad 0 \leq \beta <1, \quad |\delta|<\beta \pi/2, \quad |z|<1.
\]

In \cite{FKZ} and \cite{S87, Sug}, respectively, it is shown that
boundaries of domains $\Omega^*$ and $\Omega_s$ are a $K$-quasicircles with $K=\cot ^2(1-\beta)\pi/4$.

Setting $\Omega = \Omega^*$ or $\Omega = \Omega_s$ Theorem~A implies 
\begin{multline*}
\frac{1}{\mu_1(\Omega)} \\
{} \leq \frac{C_\alpha^2 \cot ^4(1-\beta)\frac{\pi}{4}}{\pi} \left(\frac{2 \alpha -2}{\alpha -2}\right)^{\frac{2 \alpha -2}{\alpha}}
\exp\left\{{\frac{\pi^2(2+ \pi^4)^2 \cot ^4(1-\beta)\frac{\pi}{4}}{2\log3}}\right\}\cdot \big|\Omega\big|, 
\end{multline*}
where
$$
C_\alpha=\frac{10^{6}}{[(\alpha -1)(1- \nu)]^{1/\alpha}}, \quad \nu = 10^{4 \alpha} (24\pi^2 \cot ^4(1-\beta)\pi/4)^{\alpha}\frac{\alpha -2}{\alpha -1}<1.
$$

\section{Quasiconformal reflection}
Let $\Gamma$ be a Jordan curve on the Riemann sphere, and denote its complementary components by $\Omega$, $\Omega^{*}$.
Suppose that there exists a sense-reversing quasiconformal mapping $\psi$ of the sphere onto itself  which maps $\Omega$ on $\Omega^{*}$ 
and keeps every point on $\Gamma$ fixed. Such mappings are called quasiconformal reflections.

Denote by $H$ the upper half-plane and by $H^*$ the lower half-plane of the complex plane $\mathbb C$. Consider a conformal mapping $\varphi$ of $H$ on $\Omega$ and a
conformal mapping $\varphi_*$ of $H^*$ to $\Omega^*$. Then a mapping $\varphi_*^{-1}\circ\psi\circ\varphi$ is a quasiconformal mapping of $H$ onto $H^*$ which induces a monotone mapping $h=\varphi_*^{-1}\circ\varphi$ of the real axis on itself \cite{AB}.

In \cite{AB} L. Ahlfors and A. Beurling derived a necessary and sufficient condition for a boundary mapping $h$ to be restriction of a
quasiconformal mapping of $H$ on itself (or on its reflection $H^*$). Without loss of generality it may be assumed that $h(\infty)=\infty$.
Then $h$ admits a quasiconformal extension if and only if it satisfies a $M$-condition, namely an inequality 
\begin{equation}\label{M-cond}
\frac{1}{M} \leq \frac{h(x+t)-h(x)}{h(x)-h(x-t)} \leq M
\end{equation}
which is to be fulfilled for all real $x$, $t$, $t>0$, and with a constant $M \neq 0,\infty$. More precisely,
if $h$ has a $K$-quasiconformal extension, then~\eqref{M-cond} holds with a $M=M(K)<e^{\pi K}/16$,
and if~\eqref{M-cond} holds, then $h$ has a $K$-quasiconformal extension such that $K=K(M)<M^2$. 

In \cite{Ahl} L. Ahlfors proved that a Jordan curve $\Gamma$ admits a $K$-quasiconformal reflection if and only if $\Gamma$ satisfies the Ahlfors's 3-point condition. In the following theorem, using the Ahlfors scheme,  we calculate the upper estimates for the coefficient of quasiconformality $K$ for quasicircles given by the Ahlfors's 3-point condition.

\begin{theorem}\label{Ahl}
Let a Jordan curve $\Gamma$ satisfies the Ahlfors's 3-point condition: there exists a constant $C$ such that 
\begin{equation}\label{A-cond}
|\zeta_3-\zeta_1| \leq C|\zeta_2-\zeta_1|
\end{equation}
for any three point on $\Gamma$ where $\zeta_3$ is between $\zeta_1$ and $\zeta_2$.
Then $\Gamma$ to admit a $K$-quasiconformal reflection where $K$ depends only on $C$ and
\[
K< \frac{1}{2^{10}} \exp\left\{\big(1+e^{2 \pi}C^5\big)^2\right\}.
\]
\end{theorem}

\begin{proof}

We will use the notations
\[
\alpha_1=\text{arc}\, \zeta_1\zeta_3, \quad \alpha_2=\text{arc}\, \zeta_3\zeta_2,\quad \beta_1=\text{arc}\, \zeta_2 \infty, \quad \beta_2=\text{arc}\, \zeta_1 \infty.
\]
Thus
\[
\cp(\alpha_1,\beta_1;\Omega)\cdot \cp(\alpha_2,\beta_2;\Omega)=1\,\,\text{and}\,\,\cp(\alpha_1,\beta_1;\Omega^*)\cdot \cp(\alpha_2,\beta_2;\Omega^*)=1.  
\]
%and
%\[
%\cp(\alpha_1,\beta_1;\Omega^*)\cdot \cp(\alpha_2,\beta_2;\Omega^*)=1. 
%\]
Through the conformal mapping of $\Omega$, let $\zeta_1$, $\zeta_3$, $\zeta_2$ correspond to 
$x-t$, $x$, $x+t$. This means that 
$$
\cp(\alpha_1,\beta_1;\Omega)=\cp(\alpha_2,\beta_2;\Omega)=1.
$$
Through the conformal mapping of $\Omega^*$, $\zeta_1$, $\zeta_3$, $\zeta_2$ correspond to 
$h(x-t)$, $h(x)$, $h(x+t)$.

In \cite{Ahl} Ahlfors proved that 
\begin{equation}\label{In_1}
\cp(\alpha_1,\beta_1;\Omega^*) \leq \pi \left(1+e^{2 \pi} C^5\right)^2.
\end{equation}

Now we obtain the lower estimate for the capacity $\cp(\alpha_1,\beta_1;\Omega^*)$ using the capacity estimates for the Teichm\"uller condenser:
denote by 
$$
y=(h(x+t)-h(x))/(h(x)-h(x-t)),
$$
then by Lemma~\ref{L1.3}
\begin{equation}\label{In_2}
\cp R_T(y) = \frac{2 \pi}{\log \Phi(y)}> \frac{2 \pi}{\log 32y}.
\end{equation}
Combining inequalities~\eqref{In_1} and~\eqref{In_2}, we obtain
\[
\frac{2 \pi}{\log 32y}<\cp R_T(y)=\cp(\alpha_1,\beta_1;\Omega^*) \leq \pi \big(1+e^{2 \pi}C^5\big)^2
\]
or
\[
\frac{h(x+t)-h(x)}{h(x)-h(x-t)} \leq \frac{1}{32} \exp\left\{\frac{\big(1+e^{2 \pi}C^5\big)^2}{2}\right\}.
\]
Hence $h$ satisfies an $M$-condition, i.e.
\[
\frac{1}{M(C)} \leq \frac{h(x+t)-h(x)}{h(x)-h(x-t)} \leq M(C)
\]
where 
$$
M(C)=\frac{1}{32} \exp\left\{\frac{\big(1+e^{2 \pi}C^5\big)^2}{2}\right\}.
$$
So, \cite{AB} there exist a $K$-quasiconformal reflection such that $K<M^2(C)$.

Finally we get
\begin{equation}\label{In_3}
K< \left(\frac{1}{32} \exp\left\{\frac{\big(1+e^{2 \pi}C^5\big)^2}{2}\right\}\right)^2=
\frac{1}{2^{10}} \exp\left\{\big(1+e^{2 \pi}C^5\big)^2\right\}.
\end{equation}
\end{proof}

Combining Theorem~A and Theorem~\ref{Ahl} we obtain

\begin{corollary}\label{C4.2}
Let a domain $\Omega\subset\mathbb R^2$ is bounded by a Jordan curve $\Gamma$ satisfies the Ahlfors's 3-point condition. Then
$$
\frac{1}{\mu_1(\Omega)}\leq \frac{C_\alpha^2 e^{2\left(1+e^{2\pi}C^5\right)^2}}{2^{20}\pi} \left(\frac{2 \alpha -2}{\alpha -2}\right)^{\frac{2 \alpha -2}{\alpha}}
\exp\left\{{\frac{\pi^2(2+ \pi^4)^2 e^{2\left(1+e^{2\pi}C^5\right)^2}}{2^{21}\log3}}\right\}\cdot \big|\Omega\big|, 
$$ 
for $2<\alpha<\frac{2K^2}{K^2-1}$, where
$$
C_\alpha=\frac{10^{6}}{[(\alpha -1)(1- \nu)]^{1/\alpha}}, \quad 
\nu = 10^{4 \alpha} \left(\frac{24\pi^2}{2^{20}}e^{2\left(1+e^{2\pi}C^5\right)^2}\right)^{\alpha}\frac{\alpha -2}{\alpha -1}<1.
$$
\end{corollary}

\section{Examples in fractal type domains}

\textbf{Rohde snowflake.} In \cite{Roh} S. Rohde constructed a collection $S$ of snowflake type planar
curves with the intriguing property that each planar quasicircle is bi-Lipschitz equivalent to some curve in $S$.

Rohde's catalog is 
\[
 S:= \bigcup \limits_{1/4 \leq p < 1/2} S_p
\] 
where $p$ is a snowflake parameter. Each curve in $S_p$ is built in a manner reminiscent of the construction
of the von Koch snowflake. Thus, each $S \in S_p$ is the limit of a sequence $S^n$
of polygons where $S^{n+1}$ is obtained from $S^n$ by using the replacement rule illustrated in Figure 6.1: 
for each of the $4^n$ edges $E$ of $S^n$ we have two choices, either we replace $E$ with the four line segments obtained by dividing
$E$ into four $\text{arcs}$ of equal diameter, or we replace $E$ by a similarity copy of
the polygonal $\text{arc}\, A_p$ pictured at the top right of Figure 6.1. In both cases $E$
is replaced by four new segments, each of these with diameter $(1/4) \text{diam}(E)$
in the first case or with diameter $p\, \text{diam}(E)$ in the second case. The second
type of replacement is done so that the "tip" of the replacement arc points
into the exterior of $S^n$. This iterative process starts with $S^1$ being the unit square, and the snowflake parameter,
thus the polygon $\text{arc}\, A_p$, is fixed throughout the construction.

\begin{figure}[h!]
\centering
\includegraphics[width=0.9\textwidth]{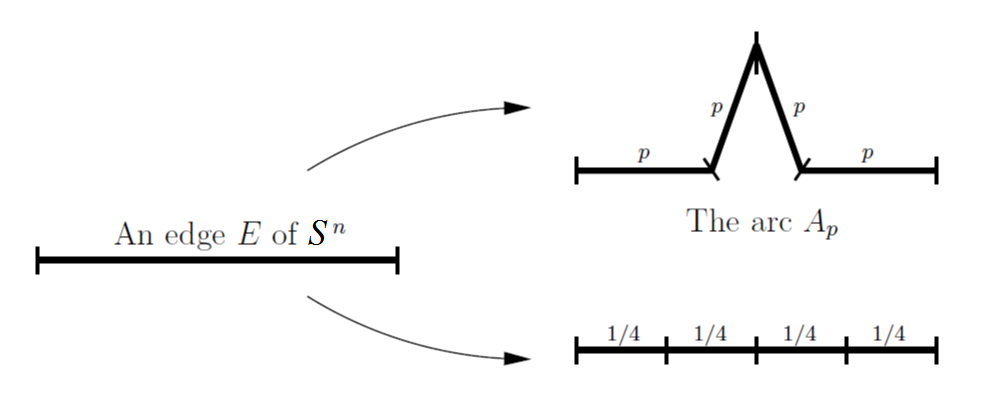}
\caption{Construction of a Rohde-snowflake.}
\end{figure}

The sequence $S^n$ of polygons converges, in the Hausdorff metric, to
a planar quasicircle $S$ that we call a \textit{Rohde snowflake} constructed with
snowflake parameter $p$. Then $S_p$ is the collection of all Rohde snowflakes
that can be constructed with snowflake parameter $p$.

In \cite{HM12} established that each Rohde snowflake $S$ in $S_p$ is $C$-bounded turning with
\[
C=C(p)=\frac{16}{1-2p}, \quad 1/4 \leq p < 1/2.
\]

A planar curve $\Gamma$ satisfies the $C$-bounded turning, $C \geq 1$, if for each pair of points $x$, $y$, on $\Gamma$,
the smaller diameter $\text{subarc}\, \Gamma[x,y]$ of $\Gamma$ that joins $x$, $y$ satisfies
\begin{equation}\label{BT-cond}
\text{diam}(\Gamma[x,y]) \leq C|x-y|.
\end{equation}   

The $C$-bounded turning condition~\eqref{BT-cond} is equivalent the Ahlfors's 3-point condition~\eqref{A-cond} with the same constant $C$ \cite{GR}.

The following theorem gives the lower estimates of the first non-trivial eigenvalue of the 
Neumann--Laplace operator in domains type a Rohde snowflakes: 

\vskip 0.3cm

{\bf Theorem C.}
\textit{Let $S_p\subset\mathbb R^2$, $1/4 \leq p < 1/2$, be the Rohde snowflake. Then the spectrum of Neumann-Laplace operator in $S_p$ is discrete, 
can be written in the form of a non-decreasing sequence
\[
0=\mu_0(S_p)<\mu_{1}(S_p)\leq\mu_{2}(S_p)\leq...\leq\mu_{n}(S_p)\leq...\,,
\]
and
\begin{multline*}
\frac{1}{\mu_1(S_p)}
\leq \frac{C_\alpha^2 e^{4\left(1+e^{2\pi}(16/(1-2p))^5\right)^2}}{2^{40}\pi} \left(\frac{2 \alpha -2}{\alpha -2}\right)^{\frac{2 \alpha -2}{\alpha}}\\
{} \times \exp\left\{{\frac{\pi^2(2+ \pi^4)^2 e^{4\left(1+e^{2\pi}(16/(1-2p))^5\right)^2}}{2^{41}\log3}}\right\} \left|S_p\right|, 
\end{multline*}
for $2<\alpha<\frac{2K^2}{K^2-1}$, where
$$
C_\alpha=\frac{10^{6}}{[(\alpha -1)(1- \nu)]^{1/\alpha}}, \quad 
\nu = 10^{4 \alpha} \left(\frac{24\pi^2}{2^{40}}e^{4\left(1+e^{2\pi}C^5\right)^2}\right)^{\alpha}\frac{\alpha -2}{\alpha -1}<1.
$$
}

The proof of Theorem C immediately follows from the Corollary~\ref{C4.2}, given that 
any $L$-bi-Lipschitz planar homeomorphism is K-quasiconformal with $K=L^2$.

\section*{Appendix}

Firstly we discuss a new notion of the hyperbolic $\alpha$-dilatation $Q(\alpha,\Omega)$ and its convergence hyperbolic interval $\HI(\Omega)$ in connection with known results. Of course, quotients 
$Q(\alpha,\Omega)$ and $\HI(\Omega)$ can be defined also in non bounded domains.

Recall definitions: 
$$
Q(\alpha, \Omega):=\iint\limits_{\mathbb D} |\varphi'(z)|^{\alpha}~dxdy=
\iint\limits_{\Omega}\left|\left( \varphi^{-1}\right)'(w)\right|^{2-\alpha}dudv
$$
where $\varphi: \mathbb D \to \Omega$ is a Riemann conformal homeomorphism;
$$
\HI(\Omega):= \left\{\alpha \in \mathbb{R}: Q(\alpha,\Omega)<\infty\right\}.
$$

Let us remark that by the definition 
$$
Q(2,\Omega)=|\Omega|,
$$
i.~e is finite for any domain of finite measure. We don't know any simple interpretation of $Q(\alpha,\Omega)$ for a number  $\alpha$ other than $2$.

In these new terms of {\bf the Inverse Brennan's conjecture} \cite{Br} states: Let $\Omega\subset\mathbb C$ be a simply connected planar domain with nonempty boundary, and $\varphi: \mathbb D\to\Omega$ be a conformal homeomorphism. Then 
$$
\HI(\Omega) =(-2,2/3).
$$
The upper bound $\alpha=2/3$ is proved, the lower bound is only conjectured. 

By \cite{GU4} for bounded domains it is proved that 
$$
\HI(\Omega) =(-2,2].
$$
The upper bound $\alpha=2$ is proved, the lower bound is only conjectured.

Note that for smooth bounded domains \cite{Br}
$$
\HI(\Omega) =(-\infty,\infty).
$$

In \cite{GU16} we demonstrated that for any $K$-quasidisc $1\le \alpha <\frac{2K^{2}}{K^2-1} \subset \HI(\Omega) $. It can be reformulated in terms of the Ahlfors condition. 

\vskip 0.3cm

\textbf{Conjecture.} Interval $\HI(\Omega)$ is defined by the hyperbolic geometry of $\Omega$.

\vskip 0.3cm

{\bf Acknowledgments}:

The first author was supported by the United States-Israel Binational Science Foundation (BSF Grant No. 2014055).

%\section{Bibliography styles}

%There are various bibliography styles available. You can select the style of your choice in the preamble of this document. These styles are Elsevier styles based on standard styles like Harvard and $%Vancouver. Please use Bib\TeX\ to generate your bibliography and include DOIs whenever available.

%Here are two sample references: \cite{Feynman1963118,Dirac1953888}.

\section*{References}

%\bibliography{mybibfile}

\vskip 0.3cm

Department of Mathematics, Ben-Gurion University of the Negev, P.O.Box 653, Beer Sheva, 8410501, Israel 
 
\emph{E-mail address:} \email{vladimir@math.bgu.ac.il} \\           
       
 Department of Higher Mathematics and Mathematical Physics, Tomsk Polytechnic University, 634050 Tomsk, Lenin Ave. 30, Russia;
 Department of General Mathematics, Tomsk State University, 634050 Tomsk, Lenin Ave. 36, Russia

 \emph{Current address:} Department of Mathematics, Ben-Gurion University of the Negev, P.O.Box 653, 
  Beer Sheva, 8410501, Israel  
							
 \emph{E-mail address:} \email{vpchelintsev@vtomske.ru}   \\
			  
	Department of Mathematics, Ben-Gurion University of the Negev, P.O.Box 653, Beer Sheva, 8410501, Israel 
							
	\emph{E-mail address:} \email{ukhlov@math.bgu.ac.il

\end{document}